\documentclass[a4paper,11pt,french,english]{amsart}
\usepackage[english]{babel}
\usepackage[utf8]{inputenc}
\usepackage{lmodern}
\usepackage[T1]{fontenc}
\usepackage{amsmath,amsthm,amssymb,amsfonts}
\usepackage{amscd}
\usepackage{pstricks}
\usepackage{pst-node}
\usepackage[a4paper,left=3.4cm,right=3.4cm,top=4cm,bottom=4cm]{geometry}
\usepackage{enumerate}
\usepackage{enumitem}
\usepackage[linktocpage=true]{hyperref}
\usepackage{url}
\usepackage{csquotes}

\newcommand{\Z}{\mathbb{Z}}

\newcommand{\N}{\mathcal{N}}
\newcommand{\C}{\mathcal{C}}
\newcommand{\F}{\mathcal{F}}

\newcommand{\cl}{\mathrm{cl}}

\newcommand{\sub}{\ensuremath{\operatorname{Sub}}}%

\renewcommand{\H}{\ensuremath{\mathcal{H}}}%
\newcommand{\A}{\ensuremath{\mathcal{A}}}%
\newcommand{\env}{\ensuremath{\mathrm{Env}}}%

\title[]{On closure operations in the space of subgroups and applications}
  \date{June 9, 2025}	

\author{Dominik Francoeur}
\address{Universidad Autónoma de Madrid, C/ Francisco Tomás y Valiente 7, 28049 Madrid, Spain}
\email{dominik.francoeur@uam.es}

\author{Adrien Le Boudec}
\address{CNRS, UMPA - ENS Lyon, 46 all\'ee d'Italie, 69364 Lyon, France}
\email{adrien.le-boudec@ens-lyon.fr}

\thanks{This work had been initiated within the framework of the Labex Milyon (ANR-10- LABX-0070) of Universite de Lyon, within the program "Investissements d'Avenir" (ANR-11-IDEX-0007) operated by the French National Research Agency (ANR)}

\theoremstyle{plain}
\newtheorem{thm}{Theorem}[section]
\newtheorem{prop}[thm]{Proposition}
\newtheorem{cor}[thm]{Corollary}
\newtheorem{lem}[thm]{Lemma}

\newtheorem{thm-intro}{Theorem}
\newtheorem{prop-intro}{Proposition}
\newtheorem{cor-intro}{Corollary}

\theoremstyle{definition}
\newtheorem{defi}[thm]{Definition}

\newtheorem{ex}[thm]{Example}
\newtheorem{rmq}[thm]{Remark}

\begin{document}

\maketitle

\begin{abstract}
We establish some interactions between uniformly recurrent subgroups (URSs) of a group $G$ and cosets topologies $\tau_\mathcal{N}$ on $G$ associated to a family $\N$ of normal subgroups of $G$. We show that when $\N$ consists of finite index subgroups of $G$, there is a natural closure operation $\H \mapsto \mathrm{cl}_\mathcal{N}(\H)$ that associates to a URS $\H$ another URS $\mathrm{cl}_\mathcal{N}(\H)$, called the $\tau_\mathcal{N}$-closure of $\H$. We give a  characterization of the URSs $\H$ that are $\tau_\mathcal{N}$-closed in terms of stabilizer URSs. This has consequences on arbitrary URSs when $G$ belongs to the class of groups for which every faithful minimal profinite action is topologically free. We also consider the largest amenable URS $\A_G$, and prove that for certain coset topologies on $G$, almost all subgroups $H \in \A_G$ have the same closure. For groups in which amenability is detected by a set of laws (a property that is variant of the Tits alternative), we deduce a criterion for $\A_G$ to be a singleton based on residual properties of $G$. 

\end{abstract}

\section{Introduction}

Let $G$ be a group\footnote{The main situation we have in mind is when $G$ is countable. Some of our results will require this countability assumption.}. We denote by $\N_G$ the set of normal subgroups of $G$. Let $\N \subseteq \N_G$ be a family of normal subgroups of $G$ that is filtering: for every $N_1,N_2 \in \mathcal{N}$ there exists $N_3 \in \mathcal{N}$ such that $N_3 \leq N_1 \cap N_2$. There is a  group topology  $\tau_\mathcal{N}$ on $G$ associated to $\N$, defined by declaring that the family of cosets $g N$, $g \in G$, $N \in \mathcal{N}$, forms a basis for $\tau_\mathcal{N}$. When $\mathcal{N}$ is the family of all  finite index normal subgroups of $G$, $\tau_\mathcal{N}$ is  the profinite topology on $G$. If $p$ is a prime and $\mathcal{N}$ is the family of finite index normal subgroups $N$ of $G$ such that $G/N$ is a $p$-group, $\tau_\mathcal{N}$ is the pro-$p$ topology.

If $H$ is a subgroup of $G$, the closure of $H$ with respect to $\tau_\mathcal{N}$ is denoted by $\mathrm{cl}_{\mathcal{N}}(H)$. In the case of the profinite topology, we use the shorter notation $\mathrm{cl}(H)$. The closure operation defines a map  \[ \mathrm{cl}_\mathcal{N} :  \sub(G) \to \sub(G), \, H \mapsto \mathrm{cl}_\mathcal{N}(H). \] Here $\sub(G)$ is the set of subgroups of $G$. That set is equipped with the topology inherited from the set $\left\lbrace 0,1\right\rbrace ^G$ of all subsets of $G$, equipped with the product topology. The space $\sub(G)$ is a compact space. The group $G$ acts on $\sub(G)$ by conjugation, and this action is by homeomorphisms. The first object of study of this article is the behaviour of the map $\mathrm{cl}_\mathcal{N}$ with respect to the dynamical system $ G \curvearrowright \sub(G)$. 

It follows from the definitions that the map $\mathrm{cl}_\mathcal{N}$ is always increasing, idempotent, and $G$-equivariant. In general $\mathrm{cl}_\mathcal{N}$ is far from being continuous. This failure of continuity already happens in the most classical case where $\tau_\mathcal{N}$ is the profinite topology. An elementary example illustrating this is the group $G = \Z[1/p]$  of $p$-adic rational numbers, for which the map $\mathrm{cl}$ is not upper semi-continuous on $\sub(G)$ (see Remark  \ref{rmq-p-adic-rationals}). Another example is $G = F_k$ (a finitely generated non-abelian free group of rank $k$). M. Hall showed that every finitely generated subgroup $H$ of $F_k$ verifies $\mathrm{cl}(H) = H$ \cite{Hall-coset-rep} (i.e.\ $F_k$ is a LERF group). Since finitely generated subgroups always form a dense subset in the space of subgroups, it follows that $\mathrm{cl}$ is the identity on a dense set of points. However $\mathrm{cl}$ is not the identity everywhere, for instance because $F_k$ admits infinite index subgroups $H$ such that $\mathrm{cl}(H) = F_k$ (e.g.\ any infinite index maximal subgroup). So $\mathrm{cl}$ is not lower semi-continuous on $\sub(F_k)$. 

The starting result of this article is that if we restrict to minimal subsystems of $\sub(G)$ (i.e.\ non-empty closed minimal $G$-invariant subsets of $\sub(G)$), the situation is better behaved. Recall that a minimal subsystem $\H \subset \sub(G)$ is called a URS (Uniformly Recurrent Subgroup) \cite{GW-urs}.   

\begin{prop-intro} \label{prop-intro-usc}
Let $\N \subseteq \N_G$ be a family of finite index normal subgroups of $G$, and let $\H$ be a URS of $G$. Then the following hold: \begin{enumerate}
	\item \label{item-usc-intro} The restriction ${\mathrm{cl}_\mathcal{N}}_{| \H}: \H \to \sub(G)$ is upper semi-continuous.
	\item \label{item-closure-intro} There exists a  unique URS contained in $\overline{\left\lbrace \mathrm{cl}_\N(H) : H \in \H \right\rbrace }$, denoted  $\mathrm{cl}_\mathcal{N}(\H)$, and called the $\tau_\N$-closure of $\H$. 
\end{enumerate}
\end{prop-intro}

The proposition  also holds in a more general situation not necessarily requiring that  $\N$ consists of finite index subgroups of $G$ (see Proposition  \ref{prop-Nclosure-usc}).

Statement (\ref{item-closure-intro}) says that there is a natural closure operation  \[ \mathrm{URS}(G) \to\mathrm{URS}(G) , \, \H \mapsto \mathrm{cl}_\mathcal{N}(\H), \] where $ \mathrm{URS}(G) $ is the set of URSs of the group $G$. We say that a URS $\H$ is closed for the topology $\tau_\N$ if $\mathrm{cl}_\N(\H) = \H$. When $G$ is a countable group, this happens if and only if there is a dense $G_\delta$-set of points $H \in \H$ such that $H$ is closed for the topology $\tau_\N$.

Recently URSs were studied and appeared in a large amount of works, including \cite{LBMB-subdyn,Bou-Houd,Fra-Gel,LBMB-growth-solv}. We refer notably to the introduction of \cite{LBMB-growth-solv} for more references. A common theme is to establish rigidity results  saying that the set of URSs of certain groups is  restricted, or to establish connections between certain group theoretic properties  of the ambient group and properties of its URSs. We believe that in certain situations the above process $\H \mapsto  \mathrm{cl}_\N(\H)$, and more generally the consideration of coset topologies on the ambient group, can be profitably used to study properties of  URSs. In Sections  \ref{sec-prof-closure} and  \ref{sec-A_G} we exhibit situations where it is indeed the case. In the remainder of this introduction we shall describe these results. 

When $\N$ consists of finite index subgroups, the property that a URS $\H$ is closed for the topology $\tau_\N$ admits the following natural characterization. Glasner--Weiss showed that to every minimal action of $G$ on a compact space $X$, there is a naturally associated URS of $G$, called the stabilizer URS of $X$, and denoted $S_G(X)$ \cite{GW-urs}. We say that the action of $G$ on a compact space $X$ is pro-$\N$ if $G \times X \to X$ is continuous, where $G$ is equipped with the topology $\tau_\N$ (see Proposition  \ref{prop-caract-pro-N-space} for characterizations of this property). 

\begin{prop-intro} \label{prop-intro-profinitelyclosed-caract}
Suppose that $G$ is a countable group and that $\N$ consists of finite index subgroups of $G$. For a URS $\H$ of $G$, the following are equivalent: \begin{enumerate}
		\item \label{item-intro-H-closed} $\H$ is closed for the topology $\tau_\N$. 
		\item \label{item-intro-H-comes-pro}  There exists a pro-$\N$ compact minimal $G$-space $X$ such that $S_G(X) = \H$. 
	\end{enumerate} 
\end{prop-intro}

 In the case of the profinite topology, the notion of pro-$\N$ $G$-space coincides with the classical notion of profinite $G$-space.  So in that situation the above proposition says that a URS $\H$ is closed for the profinite topology if and only if $\H$ is the stabilizer URS associated to a minimal profinite action of $G$. Consequences on all URSs can be drawn out of this when $G$ belongs to the class of groups for which, for a faithful minimal compact $G$-space, profinite implies topologically free.  See Proposition \ref{prop-PIF-unfaithful}. This class of groups includes non-abelian free groups, and more generally any group $G$ admitting an  isometric action on a hyperbolic space with unbounded orbits such that the $G$-action on its limit set is faithful. It also includes  hereditarily just-infinite groups. Recall that a group $G$ is just-infinite if $G$ is infinite and $G/N$ is finite for every non-trivial normal subgroup $N$, and $G$ is hereditarily just-infinite  if every finite index subgroup of $G$ is just-infinite. We call a subgroup $H$ of $G$ co-finitely dense in $G$ for the profinite topology if the profinite closure of $H$ has finite index in $G$. For hereditarily just-infinite groups we obtain:

\begin{prop-intro} \label{prop-intro-HJI}
	Let $G$ be a hereditarily just-infinite group, and let $\H$ be a non-trivial URS of $G$. Then for every $H \in \H$, $H$ is co-finitely dense in $G$ for the profinite topology.
\end{prop-intro}

In cases where we know a priori that the group $G$ has the property that the only subgroups that are co-finitely dense are the finite index subgroups, we deduce that such a group $G$ admits no continuous URS (a URS is continuous if it is not a finite set). See Corollary \ref{cor-HJI-noURS}, and the surrounding discussion for context and examples.

Another setting in which we show that the consideration of a coset topology $\tau_\N$ is fruitful with respect to the study of URSs is the case amenable URSs. A URS $\H$ is amenable if it consists of amenable subgroups. Every  group $G$ admits a largest amenable URS (with respect to a natural partial order), which is the stabilizer URS associated to the action of $G$ on its Furstenberg boundary (the largest minimal and strongly proximal compact $G$-space). This URS is denoted $\A_G$ and is called the Furstenberg URS of $G$. The action of $G$ on $\A_G$ is minimal and strongly proximal. $\A_G$ is either a singleton, in which case we have $\A_G = \left\{\mathrm{Rad}(G)\right\}$, where $\mathrm{Rad}(G)$ is the amenable radical of $G$, or $\A_G$ is continuous. We refer to \cite{LBMB-subdyn} for a more detailed discussion.

Let $\F$ denote the class of groups $G$ such $\A_G$ is a singleton. Equivalently, $G$ belongs to $\F$ if and only if every amenable URS of $G$ lives inside the amenable radical of $G$. The class $\F$ is known to be very large. It plainly contains amenable groups. It also contains all linear groups, all groups with non-vanishing $\ell^2$-Betti numbers, all hyperbolic groups, and more generally all acylindrically hyperbolic groups. We refer to \cite{BKKO} for references and details. Examples of groups outside the class $\F$ have been given in \cite{LB-C*simple}. 

The following result provides a criterion for a group to be in $\F$ that is based on residual properties of the group.  If $\C$ is a class of groups, a group $G$ is residually-$\C$ if the intersection of all normal subgroups $N$ such that $G/N \in \C$ is trivial.

\begin{thm-intro} \label{thm-intro-solvable-case}
	Let $G$ be a group such that every amenable subgroup of $G$ is virtually solvable. If $G$ is residually-$\F$, then $G$ is in $\F$. 
\end{thm-intro}

We point out that this theorem is applicable without necessarily relying on other methods related to $\F$ to verify the assumption that the group is residually-$\F$. The point is that the statement applies provided that $G$ is residually-$\C$ for some subclass $\C$ of $\F$ that is potentially much smaller. For instance the theorem applies and is already interesting if $G$ is residually finite. 

Every group satisfying the Tits alternative has the property that every amenable subgroup is virtually solvable. We refer to \S \ref{subsec-proof-Thm1} for examples of groups that are known to satisfy the Tits alternative. 

One interest of such a statement is that it is based on intrinsic algebraic properties of the group. It does not require the group $G$ to admit a rich action of geometric flavour, or to have an explicit minimal and strongly proximal compact $G$-space at our disposal. The residual properties are used as a tool in Theorem \ref{thm-intro-solvable-case}, but the confrontation of residual properties and the class $\F$ is also motivated by the fact that it is not known whether there exist residually finite groups $G$ with trivial amenable radical such that $G$ does not belong to $\F$. The groups from \cite{LB-C*simple} are never residually finite (and some of them are virtually simple). 

As an application, Theorem \ref{thm-intro-solvable-case} allows  to recover the following result from \cite{BKKO}:

\begin{cor-intro}[Breuillard--Kalantar--Kennedy--Ozawa] \label{cor-linear-F}
	If $G$ is a linear group, then $G$ is in $\F$. 
\end{cor-intro}

 The proof from \cite{BKKO} relies on linear group technology. Here the argument to deduce Corollary \ref{cor-linear-F} from Theorem \ref{thm-intro-solvable-case} uses a reduction to the case of finitely generated groups, and then only appeals to Malcev's theorem that finitely generated linear groups are residually finite, and the Tits alternative \cite{Tits72}.

The consideration of the class $\F$  is also motivated by the result of Kalantar--Kennedy that a group $G$ belongs to $\F$ if and only if the quotient of $G$ by its amenable radical is a $C^\ast$-simple group (that is, its reduced $C^\ast$-algebra is simple) \cite{KK}. We refer to the survey of de la Harpe \cite{dlHarpe} for an introduction and  historical developments on $C^\ast$-simple groups, and to the Bourbaki seminar of Raum for recent developments \cite{Raum-Bourbaki}. Hence using the result of Kalantar--Kennedy, Theorem \ref{thm-intro-solvable-case} can be reinterpreted as a criterion to obtain $C^\ast$-simplicity (under the assumption on amenable subgroups) based on residual properties of the group. See Corollary \ref{cor-Csimple}.  We are not aware of other results of this kind. 

The proof of Theorem \ref{thm-intro-solvable-case} is based on the following proposition, of independent interest. Given a group $G$, we denote by $\N_G(\F)$ the set of normal subgroups of $G$ such that $G/N \in \F$. The set $\N_G(\F)$ is stable under taking finite intersections (Lemma  \ref{lem-coset-furst-topology}), and we can consider  the coset topology on $G$ associated to $\N_G(\F)$ (and more generally to a subset $\N \subseteq \N_G(\F)$). The following result says that within the Furstenberg URS $\A_G$, almost all points have the same closure for such a topology (for technical reasons we are led to make some countability assumptions). 

\begin{prop-intro} \label{prop-intro-pro-F-closure-AG}
	Let $G$ be a countable group, and let $\N$ be a countable subset of $\N_G(\F)$. Then there exists a normal subgroup $M$ of $G$ and a comeager subset $\H_0 \subseteq \A_G$ such that $\cl_\N(H) = M$ for every $H \in \H_0$.
\end{prop-intro}

The proof of the proposition makes crucial use of the strong proximality of the action of $G$ on $\A_G$.  The proof of Theorem \ref{thm-intro-solvable-case}  is easily deduced from the proposition. The additional point is to ensure that the closed normal subgroup $M$ appearing in the conclusion of the proposition remains amenable, and this is where the two assumptions in the theorem are used. We refer to Section  \ref{sec-A_G} for details. Here we only mention that the actual setting in which we prove Theorem \ref{thm-intro-solvable-case} does not necessarily require amenable subgroups to be virtually solvable. The assumption that we need is that amenability within subgroups of $G$ can be detected by a set of laws (Definition \ref{defi-laws-detect}), a property that can be thought of as a version of the Tits alternative (Proposition \ref{prop-Tits-law-detects}). See Theorem  \ref{thm-solvable-case-bis} for the more general formulation of the theorem. 

\bigskip

\textbf{Acknowledgements.} Thanks are due to Uri Bader and Pierre-Emmanuel Caprace. We can trace back that the possibility of using Proposition \ref{prop-usc} specifically in the space of subgroups to build a  URS starting from another one and a semi-continuous map had been originally brought to our attention by them several years ago. We also tank a referee who provided various comments that improved the exposition of the paper. 

\section{Preliminaries}

A space $X$ is a $G$-space if $G$ admits a continuous action $G \times X \rightarrow X$. Throughout the paper we make the standing assumption that $G$-spaces are non-empty.  The action (or the $G$-space $X$) is \textbf{minimal} if all orbits are dense. For $x \in X$ we write $G_x$ for the stabilizer of $x$ in $G$, and $G_x^0$ for the set of  $g \in G$ such that $g$ acts trivially on a neighbourhood of $x$.  The action of $G$ on $X$ is free if $G_x = \left\lbrace 1 \right\rbrace $ for every $x \in X$, and \textbf{topologically free} if $G_x^0 = \left\lbrace 1 \right\rbrace $ for every $x \in X$.

Let $X,Y$ be compact spaces. A continuous surjective map $\pi: Y \to X$  is called \textbf{irreducible} if every proper closed subset of $Y$ has a proper image in $X$. If $X,Y$ are compact $G$-spaces and $\pi: Y \to X$ is a continuous surjective $G$-equivariant map, we say that $X$ is a factor of $Y$, and that $Y$ is an extension of $X$. When $\pi: Y \to X$ is  irreducible, we also say that $Y$ is an irreducible extension of $X$. If $\pi: Y \to X$ is  irreducible, then $X$ is minimal if and only if $Y$ is minimal. Also for $X,Y$ minimal, $\pi: Y \to X$ is irreducible if and only if it is \textbf{highly proximal}: for every $x \in X$ the fiber $\pi^{-1}(x)$ is compressible  \cite{AG-distal}.

	\subsection{Semi-continuous maps} \label{sem-cont-ext}
	
	 If $Y$ is a locally compact space, we denote by $2^Y$ the space of closed subsets of $Y$, endowed with the Chabauty  topology. The space $2^Y$ is compact.
	
	Let $X$ be a compact $G$-space. A map $\varphi \colon X \to 2^Y$ is \textbf{upper semi-continuous} if for every compact subset $K$ of $Y$, $\left\lbrace x \in X :  \varphi(x) \cap K = \emptyset \right\rbrace $ is open in $X$. It is \textbf{lower semi-continuous} if for every open subset $U$ of $Y$, $\left\lbrace x \in X :  \varphi(x) \cap U \neq \emptyset \right\rbrace $ is open in $X$. We say that $\varphi$ is semi-continuous if it is either upper or lower semi-continuous.

Let $\varphi: X \to 2^Y$ be a semi-continuous map, and $X_\varphi \subseteq X$ be the set of points where $\varphi$ is continuous. Let

\[ F_{\varphi} := \overline{ \left\{ \left( x, \varphi(x)\right)  \, : \, x \in X \right\}} \subseteq X \times 2^Y, \]

\[ E_{\varphi} := \overline{ \left\{ \left( x, \varphi(x) \right)  \, : \, x \in X_\varphi \right\}} \subseteq F_{\varphi}(X), \]

\[ T_{\varphi} := \overline{ \left\{ \varphi(x) \, : \, x \in X \right\}}, \]

\[ S_{\varphi} := \overline{ \left\{ \varphi(x) \, : \, x \in X_\varphi \right\}}. \]

We denote by $\eta : X \times 2^Y \to X$ and $p: X \times 2^Y  \to 2^Y$  the projections to the first and second coordinate. If $Y$ is second-countable, semi-continuity of $\varphi$ implies that $X_\varphi$ is a comeager subset of $X$ \cite[Theorem~VII]{Kuratowski1928}.

\begin{prop} \label{prop-usc}
Suppose $X$ is a minimal compact $G$-space, $Y$ is a locally compact $G$-space, and $\varphi: X \to 2^Y$ is $G$-equivariant and semi-continuous. Then the following hold:
\begin{enumerate}[label=(\roman*)]
	\item $F_{\varphi}$ has a unique non-empty minimal closed $G$-invariant subset $E_{\varphi}'$ , and $T_{\varphi}$ has a unique minimal closed $G$-invariant subset $S_{\varphi}'$, and $p(E_{\varphi}') = S_{\varphi}'$ .
	\item The extension $\eta : E_{\varphi}' \to X$ is highly proximal.
\end{enumerate}

If moreover $Y$ is second-countable, then $E_{\varphi}' = E_{\varphi}$ and $S_{\varphi}' = S_{\varphi}$ .
\end{prop}

\begin{proof}
See  Glasner \cite[Theorem 2.3]{Glasner-compress} and Auslander--Glasner \cite[Lemma I.1]{AG-distal}. 
\end{proof}

\subsection{The space of subgroups and URSs}

We denote by $\sub(G)$ the space of subgroups of $G$, equipped with the product topology from $\left\lbrace 0,1\right\rbrace ^G$. It is a compact $G$-space, where $G$ acts  by conjugation. When $G$ is countable, the space $\sub(G)$ is second-countable. If $H \in \sub(G)$, we denote by $H^G$ the $G$-conjugates of $H$, i.e.\ the $G$-orbit of $H$ in $\sub(G)$.

 A URS of $G$ is a (non-empty) minimal closed $G$-invariant subset of $\sub(G)$. By Zorn's lemma every (non-empty) closed $G$-invariant subset of $\sub(G)$ contains a URS.  A URS is \textbf{finite} if it is a finite $G$-orbit. A URS that is not finite is called \textbf{continuous}. By minimality and compactness, a continuous URS has no isolated points. The singleton $\left\lbrace \left\lbrace 1\right\rbrace \right\rbrace $ is called the trivial URS. If $\mathcal{P}$ is a property of groups, we say that a URS $\H$ has $\mathcal{P}$ if $H$ has $\mathcal{P}$ for every $H \in \H$. 
 
 \begin{defi}
 	If $\H$ is a URS of $G$, we denote by $\env(\H)$ the subgroup generated by all subgroups $H$ in $\H$, called the envelope of $\H$. The subgroup $\env(\H)$ is normal in $G$, and it is the smallest normal subgroup of $G$ containing some subgroup $H \in \H$.
 \end{defi}

Every minimal compact $G$-space naturally gives rise to a URS \cite{GW-urs}:

\begin{prop} \label{prop-stab-URS}
If $X$ is a compact $G$-space, then the stabilizer map $S: X \to \sub(G)$, $x \mapsto G_x$, is $G$-equivariant and upper semi-continuous. In particular if $X$ is minimal, then Proposition \ref{prop-usc} applies. 
\end{prop}

\begin{defi}
If $X$ is a minimal compact $G$-space, the unique URS contained in $\overline{ \left\{ G_x \, : \, x \in X \right\}}$ is denoted $S_G(X)$, and is called the \textbf{stabilizer URS} associated to the $G$-space $X$.
\end{defi}

One verifies that the $G$-action on $X$ is topologically free if and only if the URS $S_G(X)$ is trivial.

\begin{lem} \label{lem-containing-closed}
	Let $H,K,L$ be subgroups of $G$ such that $H \leq L$. If $K$ belongs to the closure of the $L$-orbit of $H$ in $\sub(G)$, then $K \leq L$.
\end{lem}

\begin{proof}
	The subset $\sub(L)$ is a closed subset of $\sub(G)$, and contains the $L$-orbit of $H$ since  $H \leq L$.
\end{proof}

\begin{lem} \label{lem-quotient-urs-lsc}
Let $N$ be a normal subgroup of $G$. Then the map $\sub(G) \to \sub(G)$, $H \mapsto HN$, is $G$-equivariant and lower semi-continuous.
\end{lem}

\begin{proof}
	It is $G$-equivariant because $N$ is normal in $G$. Since $G$ is discrete, lower semi-continuity means that for every $g \in G$ and every $H \in \sub(G)$ such that $g \in HN$, there is a neighbourhood of $H$ in which $g \in H'N$ remains true. If $h \in H$ is such that $g \in hN$, then the set of subgroups of $G$ containing $h$ is such a neighbourhood.
\end{proof}

\subsection{Coset topologies on groups}

Let $G$ be a group. We denote by $\N_G$ the set of normal subgroups of $G$. We make the standing convention that when considering a family  $\N \subseteq \N_G$ of normal subgroups of $G$, we always assume that $\N$ is non-empty.

\begin{defi}
Let $\mathcal{N} \subseteq \N_G$. If $E$ is a subset of $G$, we denote \[ \mathrm{cl}_{\mathcal{N}}(E) = \bigcap_{N \in \N} EN. \] 
\end{defi}

We say that $\N$ is filtering if for every $N_1,N_2 \in \mathcal{N}$ there exists $N_3 \in \mathcal{N}$ such that $N_3 \leq N_1 \cap N_2$. We record the following \cite[Chap. III]{Bourbaki-top-gen}:

\begin{prop} \label{prop-coset-topology-bourbaki}
Fix  $\mathcal{N} \subseteq \N_G$ . Then: \begin{enumerate}
	\item the family of cosets $g N$, $g \in G$, $N \in \mathcal{N}$, forms a subbasis for a group topology $\tau_\mathcal{N}$ on $G$.
		\item The topology $\tau_\mathcal{N}$ is Hausdorff if and only if $\bigcap_{ \N} N = \left\lbrace  1 \right\rbrace $. 
		\item  Suppose that $\mathcal{N}$ is filtering. Then for every subset $E$ of $G$,  the closure of $E$ with respect to $\tau_\mathcal{N}$ is equal to $\mathrm{cl}_{\mathcal{N}}(E)$.
\end{enumerate}
\end{prop}

If $\C$ is a class of groups, we denote by $\N_G(\C)$ the normal subgroups of $G$ such that $G/N \in \C$. Note that a group $G$ is residually-$\C$ if and only if $\bigcap_{ \N_G(\C)} N = \left\lbrace  1 \right\rbrace $. 

 When $\C$ is the class of all finite groups and $\N = \N_G(\C)$,  $\tau_\N$ is the profinite topology on $G$. For simplicity we write $\mathrm{cl}(H)$ for the closure in the profinite topology.  When $\C$ is the class of all finite $p$-groups ($p$ is a prime number) and $\N = \N_G(\C)$,  $\tau_\N$ is the pro-$p$ topology. In that case we write $\mathrm{cl}_p(H)$ for the closure in the pro-$p$ topology.

\subsection{Laws} \label{subsec-identity}

Let $w = w(x_1,\ldots,x_k)$ be a word in $k$ letters $x_1,\ldots,x_k$, meaning that $w$ is an element of the free group $F_k$ freely generated by $x_1,\ldots,x_k$. We assume $w$ is non-trivial. Given a group $G$, the word $w$ naturally defines a map $G^k \to G$, a $k$-tuple $(g_1,\ldots,g_k)$ being mapped to the element $w(g_1,\ldots,g_k)$ of $G$ that is obtained by replacing each $x_i$ by $g_i$. We denote by $\Sigma_w(G) \subseteq G^k$ the set of $(g_1,\ldots,g_k)$ such that $w(g_1,\ldots,g_k) = 1$.  We say that $G$  \textbf{satisfies the law} $w$ if  $\Sigma_w(G) =G^k$.

\begin{ex} \label{ex-law-solv}
Abelian groups are the groups that satisfy the law $w_1 = [x_1,x_2] =  x_1 x_2 x_1^{-1} x_2^{-1}$. More generally  let $(w_\ell)_{\ell \geq 1}$ be the sequence of laws defined inductively by the rule: if $w_\ell$ is already defined and involves $2^\ell$ variables, then \[ w_{\ell + 1}(x_{1},\ldots,x_{2^{\ell +1}}) = [w_\ell(x_1,\ldots,x_{2^{\ell}}), w_\ell(x_{2^{\ell}+1},\ldots,x_{2^{\ell +1}})].  \] Then a group $G$ is solvable of length at most $\ell$ if and only if $G$ satisfies the law $w_\ell$. 
\end{ex}

\begin{lem} \label{lem-closure-keeps-law}
	Suppose $G$ is a Hausdorff topological group, and let $w \in F_k$. Then $\Sigma_w(G)$ is a closed subset of $G^k$. In particular if a subgroup $H$ of $G$ satisfies the law  $w$, then so does its closure.
\end{lem}

\begin{proof}
Since $G$ is Hausdorff, $\left\lbrace 1 \right\rbrace $ is closed in $G$. The map $G^k \to G$ associated to $w$ being continuous, the preimage $\Sigma_w(G)$ of $\left\lbrace 1 \right\rbrace $ is a closed subset of $G^k$.
\end{proof}

\section{The $\tau_\mathcal{N}$-closure of a URS}

Let $\H$ be a closed subset of $\sub(G)$, and $L$ a subgroup of $G$. Recall that $ L^G$ is the set of $G$-conjugates of $L$. For every $\Sigma \subseteq L^G$, we write \[\H_\Sigma =  \left\lbrace H \in \H : \forall K \in L^G, \, H \subset K \Leftrightarrow K \in \Sigma \right\rbrace. \] The relevance of this definition is indicated by the following two lemmas.

\begin{lem} \label{lem-urs-disj-union}
If $\H_\Sigma \cap \H_{\Sigma'} \neq \emptyset$ then $\Sigma = \Sigma'$, and $\H$ is the disjoint union of the $\H_\Sigma$ when $\Sigma$ ranges over subsets of $L^G$. 
\end{lem}

\begin{proof}
The first assertion is consequence of the definitions. The second assertion is also clear since for every $H \in \H$, one has $H \in \H_\Sigma$ with $\Sigma =  \left\lbrace  K \in L^G : H \subset K \right\rbrace.$
\end{proof}

\begin{lem} \label{lem-urs-clopen-profinite}
Let $L$ be a subgroup of $G$ such that $L^G$ is finite. Suppose $\H$ is a URS of $G$. Then $\H_\Sigma$ is a clopen subset of $\H$ for every $\Sigma \subseteq L^G$.
\end{lem}

\begin{proof}
	For $H$ in $\H$, we let $n(H)$ be the number of conjugates of $L$ containing $H$. By our assumption, the number $n(H)$ is finite. We claim that $n(H)$ is constant on $\H$. In order to see this, take $H \in \H$ such that $n(H)  = r$ is minimal. Since not being contained in a subgroup is an open condition, one can find a neighbourhood $V$ of $H$  such that $n(H') \leq r$ for every $H' \in V$. Hence by minimality of $r$ we have $n(H') = r$ for every $H' \in V$. Now for every $K \in \H$, by minimality of the $G$-action on $\H$ the subset $V$ contains a conjugate of $K$. Since $n(K)$ is invariant under conjugation, we deduce  $n(K) = r$.

	Now fix $\Sigma \subseteq L^G$ such that $\H_\Sigma$ is non-empty, and let $H \in \H_\Sigma$. Again there is a neighbourhood $V$ of $H$ in $\H$ such that for every $H'$ in $V$, we have $H' \not\subset J$ for every $J \in L^G \setminus  \Sigma$. Moreover by the previous paragraph we have $n(H') = n(H)$. Hence by the pigeonhole principle we deduce that $H' \subset J$ for every $J \in \Sigma$. This shows that $\H_\Sigma$ is open. Since the family $(\H_\Sigma)$ forms a partition of $\H$ by Lemma \ref{lem-urs-disj-union}, it follows that $\H_\Sigma$ is also closed.
\end{proof}

\begin{defi}
Let $\mathcal{N} \subseteq \N_G$. We say that a URS $\H$ of $G$ is \textbf{$\N$-finitary} if $(HN)^G$ is finite for every $H \in \H$, $N \in \N$. 
\end{defi}

\begin{prop} \label{prop-Nclosure-usc}
Let $\mathcal{N} \subseteq \N_G$, and let $\H$ be a URS of  $G$ that is $\N$-finitary. Then the map $\H \to \sub(G)$, $H \mapsto \mathrm{cl}_\N(H)$, is upper semi-continuous. 
\end{prop}

\begin{proof}
Let $K$ be a finite subset of $G$, and let $H \in \H$ such that $\mathrm{cl}_\N(H) \cap K  = \emptyset$. One shall prove that $\mathrm{cl}_\N(H') \cap K  = \emptyset$ remains true for every $H'$ inside a neighbourhood of $H$ in $\H$. Let $g \in K$. By definition of $\mathrm{cl}_\N(H)$, there exists $N_g \in \mathcal{N}$ such that $g \notin H N_g $. Since $L = H N_g$ verifies that $(HN_g)^G$ is finite, according to Lemma \ref{lem-urs-clopen-profinite} one can find a neighbourhood $V_g$ of $H$ in $\H$ such that $H' \leq  HN_g$ for every $H' \in V_g$. A fortiori we have $H' N_g \leq  HN_g$ and hence $\mathrm{cl}_\N(H') \leq  HN_g$. Since $K$ is finite, taking the intersection over all $g \in K$ we obtain a neighbourhood $V$ of $H$ in $\H$ such that $\mathrm{cl}_\N(H') \leq \bigcap_{g \in K} HN_g$ for every $H' \in V$. Since $K$ does not intersect $\bigcap_{g \in K} HN_g$, the neighbourhood $V'$ satisfies the required property.
\end{proof}

\begin{rmq}
If $\mathcal{N}$ consists of finite index normal subgroups of $G$, then trivially every URS of $G$ is $\N$-finitary. Hence the previous proposition applies.
\end{rmq}

\begin{rmq} \label{rmq-p-adic-rationals}
Here we still consider the case where $\mathcal{N}$ consists of finite index normal subgroups of $G$, and we point out that in general the map $\sub(G) \to \sub(G)$, $H \mapsto \mathrm{cl}_\mathcal{N}(H)$, is \textit{not} upper semi-continuous. Therefore it is necessary to restrict to a URS in Proposition  \ref{prop-Nclosure-usc} in order to obtain upper semi-continuity. As an illustration, consider the group $G = \Z[1/p]$  of $p$-adic rational numbers. For $n \geq 1$, let $H_n = p^n \Z$. Then $G/H_n$ is a Prüfer $p$-group, and hence has no proper finite index subgroup. So $G$ has no proper finite index subgroup containing $H_n$, or equivalently $\mathrm{cl}(H_n) = G$. On the other hand $(H_n)$ converges to the trivial subgroup in $ \sub(G)$, which is closed for the profinite topology since $G$ is residually finite. Hence $H \mapsto \mathrm{cl}(H)$ is not upper semi-continuous. 
\end{rmq}

\begin{cor} \label{cor-closure-URS}
Let $\mathcal{N} \subseteq \N_G$, and $\H$ a URS of $G$ that is $\N$-finitary.  Then the set \[ \overline{\left\lbrace \mathrm{cl}_\N(H) : H \in \H \right\rbrace }\] contains a unique URS of $G$, that will be denoted $\mathrm{cl}_\N(\H)$. Moreover when $G$ is countable, there is a dense $G_\delta$ subset $\H_0 \subseteq \H$ such that \[ \mathrm{cl}_\N(\H) = \overline{\left\lbrace \mathrm{cl}_\N(H) : H \in \H_0 \right\rbrace }.\] 
\end{cor}

\begin{proof}
Proposition  \ref{prop-Nclosure-usc} asserts that $\H \to \sub(G)$, $H \mapsto \mathrm{cl}_\N(H)$, is upper semi-continuous. This allows to invoke Proposition  \ref{prop-usc}, from which the statement follows. 
\end{proof}

\begin{cor} \label{cor-closure-URS-closed}
	Let $\mathcal{N} \subseteq \N_G$, and $\H$ a URS of $G$ that is $\N$-finitary.  \begin{enumerate}
		\item \label{item-exists-closed-subgroup} If there exists $H \in \H$ such that $H = \mathrm{cl}_\N(H)$, then $\H = \mathrm{cl}_\N(\H)$.
		\item \label{item--closure-closed} If $G$ is countable, then $ \mathrm{cl}_\N(\mathrm{cl}_\N(\H)) = \mathrm{cl}_\N(\H)$. 
	\end{enumerate}
\end{cor}

\begin{proof}
The assumption in (\ref{item-exists-closed-subgroup}) implies that \[ \H \cap \overline{\left\lbrace \mathrm{cl}_\N(H) : H \in \H \right\rbrace } \neq \emptyset. \] So by minimality $\H$ is contained in $\overline{\left\lbrace \mathrm{cl}_\N(H) : H \in \H \right\rbrace }$. Corollary  \ref{cor-closure-URS} then implies $\H = \mathrm{cl}_\N(\H)$. (\ref{item--closure-closed}) follows from the second statement in Corollary  \ref{cor-closure-URS} and  (\ref{item-exists-closed-subgroup}). 
\end{proof}

\begin{defi}
Suppose that $\N$ is filtering, and  let $\H$ be a URS of  $G$ that is $\N$-finitary.
We say that a URS $\H$ is closed for the topology $\tau_\N$ if $\mathrm{cl}_\N(\H) = \H$. 
\end{defi}

\section{On  profinite closures of a URS} \label{sec-prof-closure}

\subsection{Profinitely closed URSs and profinite actions}

In all this section we assume that  $\mathcal{N} \subseteq \N_G$ is filtering, and that $\N$ consists of finite index subgroups of $G$. Let $\widehat{G}^\N $ be the inverse limit of the inverse system of finite groups  $G/N$, $N \in \N$, and $\psi: G \to \widehat{G}^\N$ the associated canonical group homomorphism (for simplicity we omit $\N$ in the notation in $\psi$). The group $\widehat{G}^\N $ is profinite, and $\psi: G \to \widehat{G}^\N$ is continuous, where $G$ is equipped with the topology $\tau_\N$. Recall that if $H$ is a subgroup of $G$, one has $\psi^{-1}(\overline{\psi(H)}) = \cl_\N(H)$.

\begin{prop} \label{prop-stab-urs-completion}
Let $\N$ be as above.	Then the following hold:
	\begin{enumerate}
		\item \label{item-prof-closed-uniqueURS} Let $H$ be a subgroup of $G$ such that $H = \cl_\N(H)$. Then the closure of the conjugacy class of $H$ contains a unique URS.	
		\item \label{item-prof-coset}  Let $\H$ be a URS of $G$. Let $H \in \H$, and $L = \overline{\psi(H)}$. Then the stabilizer URS associated to the left translation action of $G$ on $\widehat{G}^\N / L$ s equal to  $\cl_\N(\H)$. 
	\end{enumerate}
\end{prop}

\begin{proof}
Write $L = \overline{\psi(H)}$ and $X = \widehat{G}^\N / L$, which is a minimal compact $G$-space since $G$ has dense image in $\widehat{G}$. The stabilizer of the coset $L \in X$ in $G$ is $\psi^{-1}(L) = \cl_\N(H)$. So in case $H = \cl_\N(H)$, one has \[ \overline{H^G} \subseteq \overline{ \left\{ G_x \, : \, x \in X \right\}} . \] By Zorn's lemma $\overline{H^G}$ contains at least one URS, and it follows that is contains exactly one because  $\overline{ \left\{ G_x \, : \, x \in X \right\}} $ has this property by Proposition \ref{prop-stab-URS}. Hence (\ref{item-prof-closed-uniqueURS}) holds. 

For (\ref{item-prof-coset}), we have \[ \overline{ \left\{ G_x \, : \, x \in X \right\}} \cap \overline{ \left\{ \cl_\N(K) \, : \, K \in \H \right\}} \neq \emptyset. \] Each one of these two sets contains a unique URS, namely $S_G(X)$ and $\cl_\N(\H)$. Hence equality $S_G(X) = \cl_\N(\H)$ follows.
\end{proof}

\begin{prop} \label{prop-caract-pro-N-space}
Let $\N$ be as above, and let $X$ be a compact totally disconnected $G$-space. The following are equivalent:
	\begin{enumerate}
		\item \label{item-prof} $G \times X \to X$ is continuous, where $G$ is equipped with the topology $\tau_\N$;
		\item \label{item-clopen-finite-orb} for every clopen subset $U$ of $X$, the stabilizer of $U$ in $G$ is open for the topology $\tau_\N$;
		\item \label{item-extends-completion} $G \times X \to X$ extends to a continuous action of $\widehat{G}^\N$ on $X$.
	\end{enumerate}
If $X$ is a minimal $G$-space, these are also equivalent to:
\begin{enumerate}[resume] 
	\item there exists a closed subgroup $L$ of $\widehat{G}^\N$ such that $X$ is isomorphic to $\widehat{G}^\N / L$ as a $G$-space (where $G$ acts on $\widehat{G}^\N / L$ by left translations).
\end{enumerate}
\end{prop}

\begin{proof}
	Since $X$ is totally disconnected, clopen subsets form a basis of the topology on $X$. The equivalence between (\ref{item-prof}) and (\ref{item-clopen-finite-orb}) is therefore a consequence of the definitions. Suppose these conditions hold, and let $L$ be the closure of the image of $G$ in the group $\mathrm{Homeo}(X)$. Since it follows in particular from (\ref{item-clopen-finite-orb}) that every clopen subset of $X$ has a finite $G$-orbit, the group 
	$L$ is a profinite group. Since $G \to L$ is continuous, by the universal property of $\widehat{G}^\N$ \cite[Prop.\ 1.4.1--1.4.2]{Wilson-book}, $G \to L$ extends to a continuous homomorphism $\widehat{G}^\N \to L$. So (\ref{item-extends-completion}) holds. 	Finally  (\ref{item-extends-completion})  implies  (\ref{item-prof}) because $\psi: G \to \widehat{G}^\N$ is continuous.

	The last statement is clear since a minimal continuous action of a compact group on a compact space is necessarily transitive.
\end{proof}

\begin{defi} \label{defi-proN}
The $G$-action on $X$ is called \textbf{pro-$\N$} if it satisfies the equivalent conditions (\ref{item-prof})-(\ref{item-clopen-finite-orb})-(\ref{item-extends-completion}). We also say that the $G$-space $X$ is pro-$\N$.
\end{defi}

In case where $\N$ consists of all finite index normal subgroups of $G$, this corresponds to the common notion of profinite $G$-space. 

\begin{prop} \label{prop-URS-profinitelyclosed-caract}
Let $\H$ be a URS of a countable group $G$. Then the following are equivalent: \begin{enumerate}
	\item \label{item-H-closed} $\H$ is closed for the topology $\tau_\N$. 
	\item \label{item-H-comes-pro-comple}  For every $H \in \H$, the stabilizer URS associated to the left translation action of $G$ on $\widehat{G}^\N/\overline{\psi(H)}$ is equal to  $\H$. 
	\item \label{item-H-comes-pro}  There exists a minimal $G$-space $X$ that is pro-$\N$ such that $S_G(X) = \H$. 
\end{enumerate} 
\end{prop}

\begin{proof}
Proposition \ref{prop-stab-urs-completion} implies that (\ref{item-H-closed}) and (\ref{item-H-comes-pro-comple}) are equivalent.  (\ref{item-H-comes-pro-comple}) clearly implies (\ref{item-H-comes-pro}), so we only have to see that (\ref{item-H-comes-pro}) implies  (\ref{item-H-closed}). Let $X$ be a minimal  $G$-space that is pro-$\N$ that admits $\H$ as a stabilizer URS. By Proposition \ref{prop-caract-pro-N-space} there exists a closed subgroup $L$ of $\widehat{G}^\N$ such that $X$ is isomorphic to $\widehat{G}^\N / L$ as a $G$-space. The stabilizers in $G$ for the action on $\widehat{G}^\N / L$ are closed for the  topology $\tau_\N$ on $G$. Since $G$ is countable, there is a dense set of points $x$ in $\widehat{G}^\N / L$ such that $G_x \in S_G(\widehat{G}^\N / L)$ (Proposition  \ref{prop-stab-URS}). Since $S_G(\widehat{G}^\N / L)$ is equal to $\H$ by assumption, it follows that $\H$ contains some elements that are closed for the  topology $\tau_\N$. By Corollary  \ref{cor-closure-URS-closed} this implies $\H = \cl(\H)$. 
\end{proof}

\begin{rmq}
Matte Bon--Tsankov and Elek showed that every URS $\H$ is equal to the stabilizer URS associated to some compact $G$-space \cite{MB-Tsank,Elek}, and among the compact $G$-spaces associated to $\H$ there is a unique one that is universal in a certain sense \cite{MB-Tsank}. We point out that this $G$-space is very different from the $G$-space $\widehat{G}^\N/\overline{\psi(H)}$ associated to the specific setting considered in Proposition \ref{prop-URS-profinitelyclosed-caract}. 
\end{rmq}

\begin{rmq}
In the case of the profinite topology, Proposition \ref{prop-URS-profinitelyclosed-caract} says that a URS $\H$ is closed for the profinite topology if and only if $\H$ is the stabilizer URS associated to a minimal profinite $G$-space. It is worth noting that if $\H$ is such a URS, then the $G$-action on $\H$ need \textit{not} be profinite. Such a phenomenon has been exploited by Matte Bon \cite{MatteBon-bounded-aut} and Nekrashevych \cite{Nek-subst}. 
\end{rmq}

We end this section by showing that in general the restriction of $\cl$ to a URS is not continuous. Recall that a closed subset $F$ of a space $X$ is regular if $F$ equals the closure of its interior. 

\begin{lem} \label{lem-closure-germ-stab}
Let  $X$ be a compact minimal $G$-space such that $\mathrm{Fix}_X(g)$ is a regular closed set of $X$ for every $g \in G$. Then for every $x \in X$, we have $G_x \leq \cl(G_x^0)$. 
\end{lem}

\begin{proof}
Fix $x \in X$, $g \in G_x$, and a finite index normal subgroup $N$ of $G$. We want to see that $g \in G_x^0 N$. Since $N$ has finite index in $G$ and $G$ acts minimally on $X$, each minimal closed $N$-invariant subset of $X$ is clopen, and the minimal closed $N$-invariant subsets form a finite partition $\left\lbrace U_1, \ldots, U_n \right\rbrace $ of $X$. Let $U_i$ be the one containing $x$. Since $U_i$ is a neighbourhood of $x$ and $g \in G_x$, the assumption that $\mathrm{Fix}_X(g)$ is regular implies that there exists a non-empty open subset $V \subseteq U_i$ on which $g$ acts trivially. Since $N$ acts minimally on $U_i$, one can find $h \in N$ such that $y = hx \in V$. It follows that $g \in G_y^0 = h G_x^0h^{-1}$, and since $h \in N$ we deduce that $g \in G_x^0 N$.
\end{proof}

\begin{prop}
Suppose that $G$ is countable. Let $X$ be a minimal profinite compact $G$-space, and $\H = S_G(X)$. Suppose that $\mathrm{Fix}_X(g)$ is a regular closed set of $X$ for every $g \in G$, and the stabilizer map $X \to \sub(G)$ is not continuous on $X$. Then $\cl: \H \to \sub(G)$ is the identity on a dense set of points, but is not the identity everywhere on $\H$. In particular it is not continuous. 
\end{prop}

\begin{proof}
By Proposition \ref{prop-URS-profinitelyclosed-caract} the URS $\H$ is closed for the profinite topology. Since $G$ is countable, this means that there is a dense set of $H \in \H$ such that $\cl(H) = H$. Since $x \mapsto G_x $ is not continuous on $X$, one easily verifies that one can find $x \in X$ and $H \in \H$ such that $H \lneq G_x$ and $G_x^0 \leq H$ (see \cite[Lemma 2.8]{LBMB-subdyn}). It follows from Lemma \ref{lem-closure-germ-stab} that $G_x \leq \cl(G_x^0) \leq \cl(H)$, and hence $H$ is properly contained in $\cl(H)$ since it is properly contained in $G_x$. 
\end{proof}

\begin{rmq}
An example of the above situation is provided by $G$ the Grigorchuk group and $X$ the boundary of the defining rooted tree of $G$ \cite[Sec.\ 7]{Grigorchuk-survey}. 
\end{rmq}

\subsection{Hereditarily minimal actions}

\begin{defi}
A compact $G$-space $X$ is \textbf{hereditarily minimal} if every finite index subgroup of $G$ acts minimally on $X$. 
\end{defi}

Recall that every minimal and proximal $G$-space is hereditarily minimal \cite[Lemma 3.2]{Glasner-proxflows}.

\begin{prop} \label{prop-fi-contains-env}
Let $\H$ be a URS of $G$ that is hereditarily minimal. Then for every $H,K \in \H$ we have $\mathrm{cl}(H) =  \mathrm{cl}(K) = \mathrm{cl}(\env(\H))$.

This holds in particular if $\H = S_G(X)$ with $X$ a hereditarily minimal compact  $G$-space. 
\end{prop}

\begin{proof}
	Let $L$ be  a finite index subgroup of $G$ such that $H \leq L$. Since $L$ acts minimally on $\H$, Lemma  \ref{lem-containing-closed} says that $K \leq L$. Consequently $\mathrm{cl}(H) =  \mathrm{cl}(K)$. Since $\H$ is $G$-invariant, it follows that this common subgroup is normal in $G$. Call it $N$.  We shall see that $N = \mathrm{cl}(\env(\H))$. Since $H \leq  \env(\H)$ for every $H \in \H$, the inclusion $N \leq  \mathrm{cl}(\env(\H))$ is clear. On the other hand $N$ contains $\env(\H)$ since $N$ contains all elements of $\H$. Since $N$ is closed in the profinite topology,  $N$ contains $\mathrm{cl}(\env(\H))$. Hence equality holds.
	
	As for the last claim, it follows from the fact that Propositions \ref{prop-usc} and \ref{prop-stab-URS} ensure that the stabilizer URS associated to a hereditarily minimal compact  $G$-space is itself hereditarily minimal.	
\end{proof}

We refer to \S \ref{subsec-identity} for the definition of a law.

\begin{thm} \label{thm-env-keeps-law}
	Suppose $G$ is a residually finite group. Let $\H$ be a URS of $G$ that is hereditarily minimal, and suppose that $H \in \H$ satisfies the law $w$. Then $\env(\H)$ also satisfies the law $w$.
\end{thm}

\begin{proof}
$G$ is residually finite, so the profinite topology on $G$ is Hausdorff. Hence Lemma  \ref{lem-closure-keeps-law} says that $\mathrm{cl}(H)$ still satisfies $w$. Since $\mathrm{cl}(H)$ contains $\env(\H)$ by Proposition \ref{prop-fi-contains-env}, $\env(\H)$ also satisfies $w$.
\end{proof}

Without the hereditarily minimal assumption, it does not hold in general that a URS satisfying a law $w$ lives inside a normal subgroup of $G$ satisfying $w$, as the following example shows:

\begin{ex} \label{ex-prof-action-abelian-urs}
Let $(F_n)$ be a sequence of non-abelian finite groups. Suppose that for every $n$ there is an abelian subgroup $E_n$ of $F_n$ such that the only normal subgroup $N$ of $F_n$ containing $E_n$ is $N =F_n$. Let $\mathbb{G} = \prod_n F_n$, and let $G$ be a countable dense subgroup of $\mathbb{G}$ containing $\bigoplus_n F_n$. Consider the $G$-action on $X = \prod_n F_n / E_n$. This action is minimal and $G_x$ is abelian for every $x \in X$. In particular every $H \in \H := S_G(X)$ is abelian. On the other hand $\env(\H)$ contains the normal closure in $G$ of $\bigoplus_n E_n$. In particular $\env(\H)$  contains $\bigoplus_n F_n$, and hence $\env(\H)$ is not abelian. 

Examples as above can be found among finitely generated groups. For instance the groups constructed by B.H.\ Neumann in \cite[Ch.\ III]{Neum-37} satisfy these properties. Given a sequence $\mathbf{u} =  (u_n)_n$ of odd integers greater than $5$, B.H.\ Neumann constructs a subgroup $G_\mathbf{u}$ of the product $\prod_n \mathrm{Alt}(u_n)$. The subgroup $G_\mathbf{u}$ is generated by two elements, and $G_\mathbf{u}$ contains $\bigoplus_n \mathrm{Alt}(u_n)$ \cite[Ch.\ III (16)]{Neum-37}. Since $\mathrm{Alt}(u_n)$ is simple, $G_\mathbf{u}$ therefore falls into the above setting with $F_n =  \mathrm{Alt}(u_n)$ and $E_n$ any non-trivial abelian subgroup of $ \mathrm{Alt}(u_n)$. 
\end{ex}

\begin{rmq}
We note that in the above examples $G_\mathbf{u}$, we actually have an abelian URS $\H$ for which $\env(\H)$ satisfies no law at all (since a given non-trivial law cannot be satisfied by all finite groups, and hence neither by arbitrary large finite alternating groups). 
\end{rmq}

\subsection{PIF groups} \label{subsec-HJI}

 In this subsection we focus on the class of groups for which, for a faithful minimal compact $G$-space, profinite implies topologically free. 

\begin{defi}
We say that a group $G$ is \textbf{PIF} if for every faithful minimal compact $G$-space $X$, if the $G$-action on $X$ is profinite then it is topologically free. 
\end{defi}

This notion was studied notably by Grigorchuk. We will use the following proposition from \cite{Grigorchuk-survey}. Recall that a group $G$ is \textbf{just-infinite} (JI) if $G$ is infinite and $G/N$ is finite for every non-trivial normal subgroup $N$. Also $G$ is \textbf{hereditarily just-infinite} (HJI)  if every finite index subgroup of $G$ is  JI.

\begin{prop} \label{prop-HJI-tree-top-free}
Each one of the following conditions implies that $G$ is PIF: \begin{enumerate}
	\item \label{item-PIF} for every non-trivial subgroups $H_1,H_2 \leq G$ such that the normalizer $ N_G(H_i)$ of $H_i$ has finite index in $G$ for $i=1,2$, we have that $H_1 \cap H_2$ is non-trivial. 
	\item $G$ is hereditarily just-infinite. 
\end{enumerate}
\end{prop}

\begin{proof}
The first assertion is \cite[Proposition 4.11]{Grigorchuk-survey}  (the formulation there is not quite the same, but the argument is the same). The second assertion follows from the first one because every HJI group satisfies (\ref{item-PIF}).
\end{proof}

The first condition of the proposition is satisfied for instance by non-abelian free groups, and also by all Gromov-hyperbolic groups with no non-trivial finite normal subgroup, and more generally by any group $G$ admitting an  isometric action on a hyperbolic space $X$ with unbounded orbits such that the $G$-action on its limit set $\partial_XG$ is faithful. 

\begin{prop} \label{prop-PIF-unfaithful}
	Suppose $G$ is PIF, and let $\H$ be a non-trivial URS of $G$. Then there exists a non-trivial normal subgroup $N$ of $G$ such that $N \leq \cl(H) $ for every $H \in \H$. 
\end{prop}

\begin{proof}
	Note that since $\H$ is not the trivial URS, $\cl(\H)$ is not the trivial URS either. Let $H \in \H$, and $L = \overline{\psi(H)}$, where  $\psi: G \to \widehat{G}$ is the canonical map from $G$ to its profinite completion. By Proposition \ref{prop-stab-urs-completion}, the stabilizer URS associated to the left translation action of $G$ on $\widehat{G}/L$ is equal to  $\cl(\H)$, and hence is not trivial. This means that the action of $G$ on $\widehat{G}/L$ is not topologically free. Since this action is profinite and $G$ is PIF, the action cannot be faithful. So there  is a non-trivial normal subgroup $N$ of $G$ that is contained in the stabilizer in $G$ of the coset $L$, which is $\psi^{-1}(L) = \cl(H)$. Upper semi-continuity of $\cl$ on $\H$ then implies that $N \leq \cl(H')$ for every $H' \in \H$.
\end{proof}

\begin{defi}
We say that a subgroup $H$ of a group $G$ is \textbf{co-finitely dense for the profinite topology} if $\cl(H)$ is a finite index subgroup of $G$. 
\end{defi}

\begin{cor} \label{cor-HJI-quasidense}
Suppose $G$ is hereditarily just-infinite, and let $\H$ be a non-trivial URS of $G$. Then for every $H \in \H$, $H$ is co-finitely dense in $G$ for the profinite topology.
\end{cor}

\begin{proof}
Proposition \ref{prop-HJI-tree-top-free} says that $G$ is PIF. So Proposition \ref{prop-PIF-unfaithful} applies, and gives a non-trivial normal subgroup $N$ such that $N \leq \cl(H) $ for every $H \in \H$. By the assumption $N$ must have finite index, and the conclusion follows. 
\end{proof}

Recall that every HJI-group is either virtually simple or residually finite. Corollary \ref{cor-HJI-quasidense} is void for virtually simple groups, so the focus here is on residually finite HJI-groups. By Margulis normal subgroup theorem, every irreducible lattice $\Gamma$ in a connected semisimple Lie group $\mathbf{G}$ (with trivial center and no compact factor) of rank $ \geq 2$ is HJI. Under the assumption that every simple factor of the ambient Lie group $\mathbf{G}$ has rank $ \geq 2$, it is known that every non-trivial URS of $\Gamma$ is just the conjugacy class of a finite index subgroup \cite[Cor.\ F]{Bou-Houd}. The normal subgroup theorem of Bader--Shalom asserts that any irreducible cocompact lattice $\Gamma$ in a product $\mathbf{G_1} \times \mathbf{G_2}$, where $\mathbf{G_1}, \mathbf{G_2}$ are compactly generated topologically simple locally compact groups, is HJI \cite{Bader-Shalom}. In this setting the URSs of $\Gamma$ are not understood.

Following \cite{Cornulier-doublecoset}, we shall say that a group $G$ has \textbf{property (PF) }if for every subgroup $H$ of $G$, $H$ is co-finitely dense in $G$ for the profinite topology only if $H$ has finite index. Let $p$ be a prime number. Following \cite{EJ-PWD}, we say that a group $G$ is \textbf{weakly $p$-LERF} if for every subgroup $H$ of $G$, the closure $\cl_p(H)$ of $H$ for the pro-$p$ topology has finite index in $G$ only if $H$ has finite index. Note that applying the definition of weakly $p$-LERF to the trivial subgroup, we see that a just-infinite group that is weakly $p$-LERF is necessarily residually-$p$. Every group that is weakly $p$-LERF has property (PF).
Among their striking properties, the finitely generated groups obtained as the outputs of the process carried out in \cite{EJ-PWD} are HJI and weakly $p$-LERF.

\begin{cor} \label{cor-HJI-noURS}
Suppose $G$ is hereditarily just-infinite and has property (PF). Then $G$ has no continuous URS.
\end{cor}

\begin{proof}
Let $\H$ be a URS of $G$. If $\H$ is trivial, then there is nothing to show. Otherwise, for every  $H \in \H$, $\cl(H)$ has finite index in $G$ by Corollary \ref{cor-HJI-quasidense}. Hence so does $H$ since $G$ has (PF). In particular $H$ has only finitely many conjugates, i.e.\ $\H$ is finite.
\end{proof}

\section{The Furstenberg URS} \label{sec-A_G}

\begin{defi}
 Given two closed subsets $\mathcal{X}_1,\mathcal{X}_2 \subset \sub(G)$, we write $\mathcal{X}_1 \preccurlyeq \mathcal{X}_2$ if there exist $H_1 \in \mathcal{X}_1$ and $H_2 \in \mathcal{X}_2$ such that $H_1 \leq H_2$. 
\end{defi}

One verifies that, when restricted to the set $\mathrm{URS}(G)$, the relation $\preccurlyeq$ is a partial order \cite[Cor.\ 2.15]{LBMB-subdyn}.

Recall that a compact $G$-space $X$ is strongly proximal if the orbit closure of every probability measure on $X$ in the space $\mathrm{Prob}(X)$ contains a Dirac measure. The Furstenberg boundary $\partial_F G$ of $G$ is the universal minimal and strongly proximal $G$-space \cite{Furst-1973}, \cite{Glasner-proxflows}. We denote by $\mathrm{Rad}(G)$ the amenable radical of the group $G$. It coincides with the kernel of the action of $G$ on $\partial_F G$ (a result that holds more generally for arbitrary locally compact groups \cite{Furman-min-sp}). 

\begin{defi}
The  stabilizer URS associated to the $G$-action on $\partial_F G$  is denoted $\mathcal{A}_G$, and is called the Furstenberg URS of $G$.
\end{defi}

A result of Frol{í}k implies that the map $x \mapsto G_x$ is continuous on $\partial_F G$, so that $\mathcal{A}_G$ is exactly the collection of point stabilizers for the action of $G$ on $\partial_F G$ (see \cite{Kennedy} and references there).

\begin{prop} \label{prop-AG}
	The following hold:
	\begin{enumerate}
		\item $\mathcal{A}_G$ is amenable, and $\mathcal{X} \preccurlyeq \mathcal{A}_G$ for every non-empty closed $G$-invariant subset $\mathcal{X}$ of $\sub(G)$ consisting of amenable subgroups.
		\item $\mathcal{A}_G$ is invariant under the action of $\mathrm{Aut}(G)$ on $\sub(G)$.
		\item $\mathrm{Rad}(G) \leq H$ for every $H \in \A_G$.
				\item If $N$ is an amenable normal subgroup of $G$, and if $\sub_{\geq N}(G)$ is the set of subgroups of $G$ containing $N$, then the natural map $\varphi: \sub(G/N) \rightarrow \sub_{\geq N}(G)$ induces a $G$-equivariant homeomorphism between $\mathcal{A}_{G/N}$ and $\mathcal{A}_{G}$.
		\item $\mathcal{A}_G$ is a singleton if and only if $\A_G = \left\lbrace  \mathrm{Rad}(G) \right\rbrace $. When this does not hold, $\mathcal{A}_G$ is continuous.
		\item  $\A_G = \left\lbrace  \mathrm{Rad}(G) \right\rbrace $ if and only if every amenable URS of $G$ lives inside the amenable radical: $H \leq  \mathrm{Rad}(G) $ for every amenable URS $\H$ and every $H \in \H$.
	\end{enumerate}
\end{prop}

\begin{proof}
See \cite{LBMB-subdyn} and references there. 
\end{proof}

\begin{lem} \label{lem-N-minimal-AG}
Let $N$ be a normal subgroup of $G$ such that $H \leq N$ for every $H \in \A_G$. Then $\A_N = \A_G$. In particular $N$ acts minimally on $\A_G$.
\end{lem}

\begin{proof}
$\A_N$ being $\mathrm{Aut}(N)$-invariant, it is $G$-invariant. Hence $\A_N$ is an amenable URS of $G$. So $\A_N \preccurlyeq \A_G$. On the other hand $\A_G$ is a closed $N$-invariant subset of $\sub(N)$ consisting of amenable subgroups, so $\A_G \preccurlyeq \A_N$. Since $\preccurlyeq$ is an order in restriction to URSs, $\A_G = \A_N$.
\end{proof}

\begin{defi}
We denote by $\F$ the class of groups whose Furstenberg URS is a singleton. 
\end{defi}

\begin{lem} \label{lem-approx-furst-URS}
	Suppose $G = \bigcup_I G_i$ is the directed union of subgroups $G_i$ such that eventually $G_i$ is in $\F$ (resp. $\mathcal{A}_{G_i}$ is trivial). Then $G$ is in $\F$ (resp. $\mathcal{A}_{G}$ is trivial). 
\end{lem}

\begin{proof}
Write $R_i$ for the amenable radical of $G_i$, so that eventually $\mathcal{A}_{G_i} = \left\lbrace R_i\right\rbrace $. Consider $\varphi_i: \sub(G) \to \sub(G_i)$, $H \mapsto H \cap G_i$. This map is continuous and $G_i$-equivariant. Take $H \in \mathcal{A}_{G}$. The subset $\mathcal{X}_i := \overline{(H \cap G_i)^{G_i}}$ is a closed $G_i$-invariant subset of $\sub(G_i)$ consisting of amenable subgroups, so $\mathcal{X}_i \preccurlyeq \mathcal{A}_{G_i} = \left\lbrace R_i\right\rbrace $ by Proposition \ref{prop-AG}. Since  $\varphi_i(\A_G)$ is closed and $G_i$-invariant, $\mathcal{X}_i \subseteq \varphi_i(\A_G)$. We infer that there exists $K_i \in \A_G$ such that $K_i \cap G_i \leq R_i$. Upon passing to a subnet we may assume that $(K_i)$ converges to some $K \in \mathcal{A}_{G}$ and $(R_i)$ converges to some $R$. The subgroup $R$ is normal and amenable, so $R \leq \mathrm{Rad}(G)$. Since $G = \bigcup_I G_i$, $(K_i \cap G_i)$ also converges to  $K$, and the inclusion $K_i \cap G_i \leq R_i$ then implies $K \leq R$. So $\A_G \preccurlyeq \left\lbrace \mathrm{Rad}(G) \right\rbrace $, which means that $\A_G = \left\lbrace \mathrm{Rad}(G) \right\rbrace $ by Proposition \ref{prop-AG}. We also immediately obtain that in case $R_i$ is trivial eventually, then $\A_G$ is trivial.
\end{proof}

\subsection{Proofs of Proposition  \ref{prop-intro-pro-F-closure-AG} and Theorem \ref{thm-intro-solvable-case}} \label{subsec-proof-Thm1}

Recall that if $\C$ is a class of groups, we denote by $\N_G(\C)$ the normal subgroups of $G$ such that $G/N \in \C$. In the sequel we mainly use this notation with $\C = \F$. We have the following lemma:

\begin{lem} \label{lem-coset-furst-topology}
$\N_G(\F)$ is stable under taking finite intersections.
\end{lem}

\begin{proof}
Let $N_1, N_2 \in \N_G(\F)$. Let $Q_i = G/N_i$, and let $R_i = \mathrm{Rad}(Q_i)$ be the amenable radical of $Q_i$. Let $\pi_i: G \to Q_i$ be the canonical projection, and $M_i := \pi_i^{-1}(R_i)$. Let also $X_i = \partial_F Q_i$. The subgroup $R_i$ acts trivially on $X_i$, and the assumption that $Q_i$ belongs to $\F$ means that the $Q_i / R_i$-action on $X_i$ is free.

We consider the $G$-action on the product $X_1 \times X_2$. This action remains strongly proximal \cite[III.1]{Glasner-proxflows}. It follows that  there exists a unique minimal closed $G$-invariant subset $X \subseteq X_1 \times X_2$, and the $G$-action on $X$ is strongly proximal \cite[III.1]{Glasner-proxflows}. The subgroup $M_1 \cap M_2$ of $G$ acts trivially on $X_1 \times X_2$, and hence on $X$. Moreover since the $Q_i / R_i$-action on $X_i$ is free, it follows that for the $G$-action on $X$, every point stabilizer is equal to $M_1 \cap M_2$. Equivalently, the $G$-action on $X$ factors through a free action of $G/M_1 \cap M_2$. In particular $G/M_1 \cap M_2$ is in $\F$. Since the group $M_1 \cap M_2/ N_1 \cap N_2$ is amenable (as it embeds in the amenable group $R_1 \times R_2$), and since being in $\F$ is invariant under forming an extension with amenable normal subgroup, it follows that $G/N_1 \cap N_2$ is in $\F$.
\end{proof}

As a consequence of the lemma, it follows that the family of cosets $g N$, $g \in G$, $N \in \N_G(\F)$, forms a basis for a group topology $\tau_{\N_G(\F)}$ on $G$, and that the closure of $H$ with respect to this topology is equal to $\mathrm{cl}_{\N_G(\F)}(H)$ (Proposition  \ref{prop-coset-topology-bourbaki}).

The proof of the following is the technical part of this section. 

\begin{prop} \label{prop-co-F-subgroup}
Let $G$ be a countable group. Then for every $N \in \N_G(\F)$,  there exists a normal subgroup $M$ of $G$ such that $N \leq M$ and $M/N \leq \mathrm{Rad}(G/N)$, and a comeager subset $\H_0 \subseteq \A_G$ such that $NH = M$ for every $H \in \H_0$.
\end{prop}

\begin{proof}
	The map $\varphi_N : \A_G \to \sub(G)$, $H \mapsto NH$, is lower semi-continuous by Lemma  \ref{lem-quotient-urs-lsc}. Hence the set $\H_0 \subseteq \A_G$ of points where $\varphi_N$ is continuous is a comeager subset of $\A_G$, and
	
	\[ E_{\varphi_N} = \overline{ \left\{ \left( H, NH  \right)  \, : \, H  \in \H_0 \right\}} \, \text{ and } \,  
	S_{\varphi_N} = \overline{ \left\{ NH \, : \, H  \in \H_0 \right\}} \] satisfy the conclusion of Proposition \ref{prop-usc}. Note that  $S_{\varphi_N}$ is contained in the closed subset $\sub_{\geq N}(G)$ of $\sub(G)$ consisting of subgroups of $G$ containing $N$.

	Since the URS $\A_G$ is strongly proximal and strong proximality passes to highly proximal extensions and factors \cite[Lemma 5.2]{Glasner-compress}, we deduce that the $G$-action on  $S_{\varphi_N}$ is minimal and strongly proximal. The map $\pi_N : S_{\varphi_N} \to \sub(G/N)$, $K \to K/N$, is a $G$-equivariant homeomorphism onto its image (indeed, one easily verifies that modding out by $N$ defines a homeomorphism from  $\sub_{\geq N}(G)$ onto $\sub(G/N)$). Hence $\pi_N(S_{\varphi_N})$ is a strongly proximal URS of $G/N$. Moreover $\pi_N(S_{\varphi_N})$  consists of amenable subgroups. If we let $R$ be the amenable radical of $G/N$, it follows from the assumption that $G/N$ belongs to $\F$ that $\pi_n(S_{\varphi_N})$ is contained in $\sub(R)$. On the other hand $R$ must act trivially on $\pi_N(S_{\varphi_N})$ by minimality and strong proximality. 
	
	Consider the envelope $E = \env(S_{\varphi_N})$. Since $\pi_N(S_{\varphi_N}) \subseteq \sub(R)$, we have $E /N \leq R$. So by the previous paragraph and the fact that $\pi_N : S_{\varphi_N} \to \pi_N(S_{\varphi_N})$ is a $G$-equivariant homeomorphism, it follows that $E$ acts  trivially on $S_{\varphi_N}$. On the other hand $E$ contains $H$ for every $H \in \H_0$, and since $\H_0$ is dense in $\A_G$ this easily implies that $E$ contains $H$ for every $H \in \A_G$. Hence Lemma \ref{lem-N-minimal-AG} can be applied to $E$, and we infer that $E$ acts minimally on $\A_G$. Using Proposition \ref{prop-usc} and the fact that minimality passes to irreducible extensions, we deduce that $E$ acts minimally on $S_{\varphi_N}$. All together, this shows that $S_{\varphi_N}$ is a one-point space. The corresponding normal subgroup $M$ of $G$ verifies the conclusion. 
\end{proof}

We  deduce Proposition  \ref{prop-intro-pro-F-closure-AG} from the introduction, that we state again for the reader's convenience:

\begin{prop} \label{prop-pro-F-closure-AG}
	Let $G$ be a countable group, and let $\N$ be a countable subset of $\N_G(\F)$. Then there exists a normal subgroup $M$ of $G$ and a comeager subset $\H_0 \subseteq \A_G$ such that $\cl_\N(H) = M$ for every $H \in \H_0$.
\end{prop}

\begin{proof}
	We apply Proposition \ref{prop-co-F-subgroup} for every $N \in \N$. We obtain a normal subgroup $M_N$ of $G$ and a comeager subset  $\H_N$  of $\A_G$.	Set $\H_0 = \bigcap_\N \H_N$ and $M = \bigcap_\N M_N$. Since $\N$ is countable, $\H_0$ is a comeager subset of $\H$. By construction for every $H \in H_0$, \[cl_\N(H) = \bigcap_\N NH =  \bigcap_\N M_N = M. \]
\end{proof}

\begin{defi} \label{defi-laws-detect}
	We say that a set of laws $\mathbb{W}$ \textbf{detects amenability} in a group $G$ if for every subgroup $H$ of $G$, one has that  $H$ is amenable if and only if there exists $w \in \mathbb{W}$ such that $H$ virtually satisfies $w$.
\end{defi}

Note that if a set of laws detects amenability in a group $G$, it also detects amenability in any subgroup of $G$.

Recall that a group $G$ satisfies the Tits alternative if every subgroup of $G$ is either virtually solvable or contains a non-abelian free subgroup. Beyond the original result of Tits about linear groups \cite{Tits72}, many classes of groups are known to satisfy the Tits alternative. Examples include hyperbolic groups \cite{Gromov-hyp}, the outer-automorphism group of a finitely generated free group \cite{BFH-1,BFH-2} and of many free products \cite{Horbez}, groups acting properly on a finite-dimensional CAT(0) cube complex and with finite subgroups of bounded order \cite{Sageev-Wise}, the  group of polynomial automorphisms of $\mathbb{C}^2$ \cite{Lamy}, the Cremona group of birational transformations of the projective space $\mathbb{P}^2(\mathbb{C})$ \cite{Cantat,Urech}. 


\begin{prop} \label{prop-Tits-law-detects}
If a group $G$ satisfies the Tits alternative, then there is a set of laws that detects amenability in $G$.
\end{prop}

\begin{proof}
For $\ell \geq 1$, let $w_\ell$ be the law from Example \ref{ex-law-solv}. So a group is solvable of length at most $\ell$ if and only if it satisfies $w_\ell$. Let $H$ be a subgroup of $G$. Since $G$ satisfies the Tits alternative, $H$ is amenable if and only if $H$ is virtually solvable, if and only if there is $\ell$ such that $H$ virtually satisfies $w_\ell$. So $\mathbb{W} = \left\lbrace w_\ell \right\rbrace_{\ell \geq 1}$ detects amenability in $G$.
\end{proof}

The following is an immediate consequence of the definition and Lemma \ref{lem-closure-keeps-law}.

\begin{lem} \label{lem-closure-remains-amen}
	Let $G$ be a group such that there is a set of laws that detects amenability in $G$, and let $(G,\tau)$ be a group topology on $G$ that is Hausdorff. Then for every amenable subgroup $H$ of $G$, the $\tau$-closure of $H$ remains amenable. 
\end{lem}

We note that in view of Proposition \ref{prop-Tits-law-detects}, the following result covers Theorem \ref{thm-intro-solvable-case} from the introduction. And this result applies to all the groups (and to all their subgroups) satisfying the Tits alternative mentioned above. 

\begin{thm} \label{thm-solvable-case-bis}
Let $G$ be a group such that there is a set of laws that detects amenability in $G$. Then $G$ is in $\F$ if and only if $G$ is residually-$\F$.
\end{thm}

\begin{proof}
Only one direction is non-trivial. Writing $G$ as the directed limit of its countable subgroups and invoking Lemma \ref{lem-approx-furst-URS}, one sees that it suffices to prove the result when $G$ is countable. Under this assumption, since $G$ is residually-$\F$, one can find $\N \subseteq \N_G(\F)$ such that $\N$ is countable and $\bigcap_\N N = \left\lbrace 1\right\rbrace $. Since $\N_G(\F)$ is stable under taking finite intersections by Lemma \ref{lem-coset-furst-topology}, we can replace $\N$ by the collection of finite intersections of elements of $\N$, so that we may assume that $\N$ is filtering. So for a subgroup $H$ of $G$, $\cl_\N(H) $ equals the closure of $H$ in the topology $\tau_\N$ (Proposition  \ref{prop-coset-topology-bourbaki}). 
	
	Proposition \ref{prop-pro-F-closure-AG} provides a normal subgroup $M$ of $G$ such that $\cl_\N(H) = M$ for every $H$ in a comeager subset of $\A_G$. The topology $\tau_\N$ is Hausdorff since $\bigcap_\N N = \left\lbrace 1\right\rbrace $, so it follows from Lemma \ref{lem-closure-remains-amen} that the closure of an amenable subgroup of $G$ remains amenable. This shows $M$ is amenable, and it follows that $M \leq \mathrm{Rad}(G)$. By Proposition \ref{prop-AG} this means that $\A_G = \left\lbrace \mathrm{Rad}(G) \right\rbrace $. 
\end{proof}

\begin{rmq}
When the group $G$ is residually finite, there is a shorter way to obtain the conclusion of Theorem  \ref{thm-solvable-case-bis}. Indeed, since the $G$-space $\partial_F G$ is proximal, it is hereditarily minimal \cite[Lemma 3.2]{Glasner-proxflows}. Moreover it follows from the conclusion of  Proposition \ref{prop-usc} that being hereditarily minimal is inherited from a $G$-space to its stabilizer URS. Hence $\A_G$ is a hereditarily minimal URS. Hence Proposition \ref{prop-fi-contains-env} and Theorem \ref{thm-env-keeps-law} apply, and the conclusion follows as above. 
\end{rmq}

\begin{cor} \label{cor-Csimple}
	Let $G$ be a group such that there is a set of laws that detects amenability in $G$, and suppose $\mathrm{Rad}(G)$ is trivial. If $G$ is residually-$\F$, then $G$ is $C^\ast$-simple. 
\end{cor}

\begin{proof}
The result follows from Theorem  \ref{thm-solvable-case-bis} and the main result of \cite{KK}, which asserts that $G$ is in $\F$ if and only if $G/ \mathrm{Rad}(G)$ is $C^\ast$-simple. 
\end{proof}

\subsection{Linear groups}

We deduce Corollary \ref{cor-linear-F} from the introduction, which asserts that linear groups belong to $\F$.

\begin{proof}[Proof of Corollary \ref{cor-linear-F}]
Writing $G$ as the directed limit of its finitely generated subgroups and invoking Lemma \ref{lem-approx-furst-URS}, one sees that without loss of generality we can assume that $G$ is a finitely generated linear group. By Malcev's theorem, the group $G$ is residually finite. Also by the Tits alternative \cite{Tits72}, every amenable subgroup of $G$  is virtually solvable (we are using again that $G$ is finitely generated to have this version of the Tits alternative). Hence all the assumptions of Theorem  \ref{thm-solvable-case-bis} are verified. The conclusion follows.
\end{proof}

\bibliographystyle{amsalpha}
\bibliography{urssandcosettopologies}

\end{document}